\documentclass{amsart}
\usepackage{verbatim}
\usepackage{rotating} %for \includegraphix[angle=270]
\usepackage{amssymb}
\newtheorem{theorem}{Theorem}[section]

\newtheorem{prop}[theorem]{Proposition}

\newtheorem{lemma}[theorem]{Lemma}

\newtheorem{cor}[theorem]{Corollary}
\newtheorem{question}[theorem]{Question}
\newtheorem{definition}[theorem]{Definition}

\newtheorem{conjecture}[theorem]{Conjecture}

\theoremstyle{plain}
\numberwithin{equation}{theorem}

\theoremstyle{remark}

\newtheorem{remark}[theorem]{Remark}

\newcommand{\C}{{\mathbb C}}
\newcommand{\F}{{\mathbb F}}
\newcommand{\Fp}{{\mathbb F}_p}
\newcommand{\Q}{{\mathbb Q}}

\newcommand{\N}{{\mathbb N}}

\newcommand{\bP}{{\mathbb P}}

\newcommand{\bA}{{\mathbb A}}

\newcommand{\lra}{\longrightarrow}
\newcommand{\cO}{\mathcal{O}}

\begin{document}

\title[Dynamical Mordell-Lang Problem]{The Dynamical Mordell-Lang problem}

\author{J.~P.~Bell}
\address{
Jason Bell\\
Department of Pure Mathematics\\
University of Waterloo\\
Waterloo, ON N2L 3G1\\
Canada 
}
\email{jpbell@uwaterloo.ca}

\author{D.~Ghioca}
\address{
Dragos Ghioca\\
Department of Mathematics\\
University of British Columbia\\
Vancouver, BC V6T 1Z2\\
Canada
}
\email{dghioca@math.ubc.ca}

\author{T.~J.~Tucker}
\address{
Thomas Tucker\\
Department of Mathematics\\
University of Rochester\\
Rochester, NY 14627\\
USA
}
\email{ttucker@math.rochester.edu}

\begin{abstract}
Let $X$ be a Noetherian space, let $\Phi:X\lra X$ be a continuous function, let $Y\subseteq X$ be a closed set, and let $x\in X$. We show that the set $S:=\{n\in\N\colon \Phi^n(x)\in Y\}$ is a union of at most finitely many arithmetic progressions along with a set of Banach density zero. In particular, we obtain that given any quasi-projective variety $X$, any rational map $\Phi:X\lra  X$, any subvariety $Y\subseteq X$, and any point $x\in X$ whose orbit under $\Phi$ is in the domain of definition for $\Phi$, the set $S$ is
 a finite union of arithmetic progressions together with a set of Banach density zero. We prove a similar result for the backward orbit of a point.
\end{abstract}

\maketitle

\section{Introduction}

The classical Mordell-Lang question (solved by Faltings \cite{Faltings} in the case of abelian varieties and by  Vojta \cite{Vojta} in the case of semiabelian varieties) asks for a description of the algebraic relations between points in a given finitely generated subgroup of a semiabelian variety $G$ defined over $\C$. More precisely, the Mordell-Lang problem asks that the intersection between any subvariety $Y$ of $G$ with a finitely generated subgroup $\Gamma$ of $G(\C)$ is a finite union of subgroups of $\Gamma$. Therefore, Faltings and Vojta's results assert that there exists a robust \emph{Mordell-Lang principle}   which governs the geometry of a semiabelian variety.

If we interpret $\Gamma$ as the image of the identity $0\in G(\C)$ under a finitely generated subgroup of translations on $G$, then we obtain a reformulation of the Mordell-Lang question in the context of algebraic dynamics. One may further generalize and ask for a description of the intersection between any subvariety $Y$ of $G$ with the orbit of any point $\alpha\in G(\C)$ under a finitely generated commutative semigroup $S$ of endomorphisms of $G$. However, this generalization turns out to have surprising answers; in particular, if the endomorphisms from $S$ are group homomorphisms, then the Mordell-Lang principle applies only if each endomorphism from $S$ has diagonalizable Jacobian at the identity of $G$ (see \cite{JNT-Bordeaux} for a full discussion of this problem). However, if we consider the case of a cyclic semigroup $S$, i.e. $S$ is generated by a single endomorphism of $G$, then the Mordell-Lang principle holds for semiabelian varieties, and furthermore it is expected to hold in much higher generality. The following Conjecture was formulated in \cite{GT-JNT}.

\begin{conjecture}
\label{DML}
Let $X$ be any quasi-projective variety defined over $\C$, let $Y$ be any subvariety of $X$, let $\alpha\in Y(\C)$, and let $\Phi$ be an endomorphism of $X$. Then the set of $n\in\N$ such that $\Phi^n(\alpha)\in Y(\C)$ is a finite union of arithmetic progressions.
\end{conjecture}

We note that for an arithmetic progression we allow the possibility that its ratio equals $0$ (in which case it is constant); so, in Conjecture~\ref{DML} we allow the possibility of a finite intersection (which often is the case). 
Also we denote by $\cO_\Phi(\alpha)$ the orbit of $\alpha$ under $\Phi$, i.e. the set consisting of all $\Phi^n(\alpha)$ for nonnegative integers $n$ (as always in algebraic dynamics, we denote by $\Phi^n$ the $n$-th iterate of $\Phi$). In case of a rational self-map $\Phi$ we always work under the hypothesis that $\cO_\Phi(\alpha)$ is entirely contained in the domain of definition for $\Phi$.

A reformulation of Conjecture~\ref{DML} would be that if $Y$ is a subvariety of $X$ which contains infinitely many points of the form $\Phi^n(\alpha)$ (with $n\in\N$), then $Y$ must contain a positive dimensional periodic subvariety under the action of $\Phi$, which has nontrivial intersection with $\cO_{\Phi}(\alpha)$. This statement is in line with the classical Mordell-Lang problem, since in that case, if a subvariety $Y$ of a semiabelian variety contains infinitely many points from a finitely generated subgroup $\Gamma$, then $Y$ must contain a translate of an algebraic subgroup of positive dimension which has nontrivial intersection with $\Gamma$.

Conjecture~\ref{DML} was proved in \cite{AJM} for all \'{e}tale endomorphisms $\Phi$ of any quasi-projective variety $X$. The proof relies on constructing a $p$-adic analytic function which parametrizes the orbit of $\alpha$ under $\Phi$. This idea was pioneered by one of the authors in \cite{JPB} (see also \cite{Bell-Lagarias} for an extension of this method to orbits of subvarieties under an automorphism of an affine variety). We also note a similar construction of a $v$-adic parametrization for orbits of points under Drinfeld modules (done by two of the authors in \cite{GT-Compositio}). A crucial ingredient in constructing this $p$-adic analytic map is the fact that the orbit does not intersect the ramified locus for $\Phi$.  
However, the case of ramified self-maps $\Phi$ remains open in  general. Only special instances of Conjecture~\ref{DML} when the map $\Phi$ is ramified are known (see \cite{BGKTZ, BGHKST, GTZ-1, Xie} and also Wang \cite{Wang} for an extension of the Dynamical Mordell-Lang problem to analytic endomorphisms of the unit disk). In almost all known ramified cases, $\Phi$ is given by the coordinatewise action through one-variable rational maps on $(\bP^1)^m$, i.e. $\Phi(x_1,\dots,x_m)=(\varphi_1(x_1),\dots,\varphi_m(x_m))$ for some rational maps $\varphi_i$. In particular, very little is known for arbitrary endomorphisms of quasi-projective varieties, besides the result of Fakhruddin \cite{Fakhruddin}, who proved that the Dynamical Mordell-Lang Conjecture holds for \emph{generic}  endomorphisms of $\bP^n$. In this paper we obtain a very general result for Noetherian spaces in support of Conjecture~\ref{DML}. First we recall the definition of Banach density for subsets of $\N$, and then we define Noetherian topological spaces.

\begin{definition}
{\em Let $S$ be a subset of the natural numbers.  We define the \emph{Banach density} of $S$ to be
$$\delta(S):=\limsup_{|I|\to \infty} \frac{|S\cap I|}{|I|},$$ where the $\limsup$ is taken over intervals $I$ in the natural numbers. We say that a subset $S$ of the natural numbers has \emph{Banach density zero} if $\delta(S)=0$.}
\end{definition}

\begin{definition}
{\em Let $X$ be a topological space. We say that $X$ is \emph{Noetherian} if it satisfies the descending chain condition for its closed subsets, i.e., there exists no infinite descending chain of proper closed subsets. }
\end{definition}

\begin{theorem}
\label{main result}
Let $X$ be a Noetherian topological space, and let $\Phi:X\lra X$ be a continuous function. Then for each $x\in X$ and for each closed subset $Y$ of $X$, the set $S:=\{n\in\N\colon \Phi^n(x)\in Y\}$ is a union of at most finitely many arithmetic progressions along with a set of Banach density zero.
\end{theorem}

In particular, Theorem~\ref{main result} implies the following Corollary which provides evidence to an extension of the Dynamical Mordell-Lang Conjecture to the case of rational maps.  
\begin{cor}
\label{rational maps corollary}
Let $X$ be a quasi-projective variety defined over a field $K$, let $\Phi:X\lra X$ be a rational map defined over $K$, let $x\in X(K)$ such that $\cO_\Phi(x)$ is contained in the domain of definition for $\Phi$, and let $Y$ be a $K$-subvariety of $X$. Then the set $S:=\{n\colon \Phi^n(x)\in Y(K) \}$ is a union of at most finitely many arithmetic progressions along with a set of Banach density zero.
\end{cor}

Theorem~\ref{main result} yields that the \emph{Dynamical Mordell-Lang principle} almost holds for all continuous self-maps on Noetherian topological spaces. But obviously in this great generality it cannot always hold; we already know that if $X$ is a $p$-adic analytic manifold and $\Phi$ is an analytic homomorphism, then the set $S$ might be infinite without containing an infinite arithmetic progression (see \cite[Proposition 7.1]{gap-Compo}).   
Theorem~\ref{main result} shows that once removing finitely many arithmetic progressions contained in $S$, we obtain a very sparse set. The key for our proof is the following Proposition which we state in the context of endomorphisms of quasi-projective varieties, but it is true in the more general context of continuous maps on  Noetherian spaces (see Proposition~\ref{prop: ap}).

\begin{prop} Let $X$ be a quasi-projective variety defined over the field $K$, let $x\in X(K)$, and let $\Phi : X\to X$ be an endomorphism defined over $K$.  Assume $Y$ is a Zariski closed subset of $X$ with the property that the set $S:=\{n\colon \Phi^n(x)\in Y\}$ has positive Banach density.  Then $S$ contains an infinite arithmetic progression.  
\label{prop: ap nice}
\end{prop}

Similar, but weaker, results were previously obtained in \cite{Denis} and \cite{gap-Compo}. Denis \cite{Denis} has treated the question of the distribution of the set $S$ when $Y$ does not contain a
periodic subvariety. He showed, for any morphism of varieties over a field of characteristic $0$, that
$S$ cannot be \emph{very dense of order 2} (see \cite[D\'efinition
2]{Denis}); this is a weaker conclusion than being of Banach density
$0$ (which is the result of our Proposition~\ref{prop: ap nice}). On
the other hand, the result of \cite{gap-Compo} yields a stronger
statement in terms of the sparseness of the set $S$ (assuming this set
does not contain an infinite arithmetic progression); however the
result of \cite{gap-Compo} applies only to
endomorphisms of $(\bP^1)^n$ of the form $(\varphi_1,\dots,\varphi_n)$
where each $\varphi_i$ is a rational map defined over a field of
characteristic $0$.

We note that due to the general setting of our Theorem~\ref{main result} we are able to prove Corollary~\ref{rational maps corollary} for all rational maps on a quasi-projective variety, as opposed to regular maps only. We also note that our results apply  in positive characteristic, in which case one knows that the Dynamical Mordell-Lang Conjecture fails (since also the classical Mordell-Lang principle fails in positive characteristic). For example, if $G$ is a semiabelian variety defined over $\Fp$, $C\subset G$ is a curve of genus greater than $1$ defined over $\Fp$, and $\gamma\in C$ is a (generic) point not defined over $\overline{\F}_p$, then the intersection of $C$ with the cyclic subgroup of $G$ generated by $\gamma$ consists of all points of the form $p^n\gamma$ for nonnegative integers $n$. So, in positive characteristic, Proposition~\ref{prop: ap nice} is (in some sense) best possible because the set $S$ might consist of all powers of the prime number $p$ only. We note that when $K$ has positive characteristic, Ghioca and Scanlon conjectured the precise form of a set $S$ from Proposition~\ref{prop: ap} (in case it does not contain an infinite arithmetic progression). It is expected in that case $S$ consists of finitely many sets of the form 
\begin{equation}
\label{F-sets}
\left\{\sum_{i=1}^k c_i p^{\ell_i n_i}\colon n_i\in\N\right\} 
\end{equation}
for given $k\in\N$, $c_i\in\Q$ and $\ell_i\in\N$. For example, let $C$ be a curve of high genus (at least equal to $3$)  defined over $\Fp$ embedded in its Jacobian $J$, and let $\alpha\in C(\Fp(t))\setminus C(\Fp)$. Then the Zariski closure $V$ of $C+C$ (where the addition takes place inside $J$) \emph{generically} does not contain a translate of a positive dimensional algebraic group. So, the intersection of $V$ with the orbit of $0\in J$ under the translation-by-$\alpha$ map consists of points of the form $(p^m+p^n)\cdot \alpha$ only. In conclusion, if $\Phi$ is an endomorphism of a quasi-projective variety, assuming the set $S$ from Proposition~\ref{prop: ap nice} does not contain an infinite arithmetic progression, if $K$ has characteristic $0$, then it is expected (according to Conjecture~\ref{DML}) that $S$ is finite, while if $K$ has positive characteristic, then it is expected to consist of finitely many subsets of the form \eqref{F-sets}. However, proving both these two precise results seems very difficult at this moment.

Using a similar approach to that of Theorem~\ref{main result} we are
able to prove a result similar to Theorem~\ref{main result} for the
backward orbit of a point in a Noetherian space. More precisely, for a
Noetherian space $X$, a continuous function $f:X\lra X$, and a point
$x\in X$, we define a \emph{coherent backward orbit of $x$} (with
respect to $f$) be a sequence $\{x_{-n}\}_{n\ge 0}$ such that
$$x_0=x\text{ and }f(x_{-n-1})=x_{-n}\text{ for each }n\ge 0.$$
We obtain the following result.
\begin{theorem}
\label{main result DMM}
Let $X$ be a Noetherian space, let $f:X\lra X$ be a continuous function, let $\{x_{-n}\}_{n\ge 0}$ be a coherent backward orbit of a point  $x\in X$, and let $Y\subseteq X$ be a closed set. Then the set $S:=\{n\in\N\colon x_{-n}\in Y\}$ is a union of at most finitely many arithmetic progressions and a set of Banach density zero.
\end{theorem}
In particular, we ask the following question for algebraic dynamical systems.
\begin{question}
\label{DMM}
Let $X$ be a quasi-projective variety defined over $\C$, let $\Phi:X\lra X$ be an endomorphism, let $\{x_{-n}\}_{n\ge 0}$ be a coherent backward orbit of a point $x\in X(\C)$ (with respect to $\Phi$), and let $Y\subseteq X$ be a subvariety. Is the set $S:=\{n\in\N\colon x_{-n}\in Y(\C)\}$ a union of at most finitely many arithmetic progressions?
\end{question}
Question~\ref{DMM} is related to the Dynamical Manin-Mumford Conjecture (see \cite{ZhangLec, GT-IMRN}). A positive answer to Question~\ref{DMM} yields that if a subvariety $Y$ contains a Zariski dense set of points in common with a coherent backward orbit of a point $x\in X(\C)$, then $Y$ is periodic under the action of $\Phi$. In the original formulation (see \cite{ZhangLec}) of the Dynamical Manin-Mumford Conjecture, it was asked whether a subvariety $Y$ of a projective variety $X$ would have to be preperiodic under the action of a polarizable endomorphism $\Phi$ of $X$ if $Y$ contains a Zariski dense set of preperiodic points (we call a point $y\in X$ preperiodic if its orbit $\cO_\Phi(x)$ is finite). If the point $x$ in Question~\ref{DMM} is preperiodic, then each point in a coherent backward orbit of $x$ is preperiodic, and thus a positive answer to Question~\ref{DMM} provides a positive answer to this special case of the Dynamical Manin-Mumford Conjecture. We note that there are counterexamples coming from endomorphisms of CM abelian varieties  to the original formulation of the Dynamical Manin-Mumford Conjecture (see \cite{GT-IMRN, Fabien}), but we do not know whether those types of counterexamples can be found to Question~\ref{DMM}.

While writing this paper we found out that, using different
techniques, Clayton Petsche proved Theorem~\ref{main result} when
$\Phi$ is an endomorphism of an affine variety $X$. Petsche uses methods from topological dynamics and
ergodic theory; in particular, he uses Berkovich spaces and a strong topological version of the
Poincar\'e Recurrence Theorem.  Also, William
Gignac indicated to us that both Theorem~\ref{main result} and
\ref{main result DMM} follow using arguments stemming from a
deep result of ergodic theory on Noetherian spaces proven by Charles Favre (this is Th\'eor\`eme~2.5.8 in Favre's PhD thesis); see also \cite[Theorem 1.6]{Gignac} for an alternative
proof of Favre's result using measure-theoretic arguments.  We thank
both Clayton Petsche and William Gignac for useful conversations on
this topic. An advantage of our direct approach to proving Theorem~\ref{main result} (and the related results) is that we derive a simple result regarding subsets of $\N$ of positive Banach density (see Lemma~\ref{lem}) which allows us to derive concrete quantative results (see Theorem~\ref{the case of curves} and also Remark~\ref{remark reccursive formula}).

We now briefly sketch the plan of this paper. In Section~\ref{warm-up}, we
prove the key technical Lemma~\ref{lem}. In Section~\ref{proof of our
  main result} we prove Proposition~\ref{prop: ap} (which is
essentially Proposition~\ref{prop: ap nice} for Noetherian spaces) and
then deduce various consequences such as Theorems~\ref{main result}
and \ref{main result DMM}.  We conclude with some remarks on
quantitative results, including Theorem~\ref{the case of curves} in
Section~\ref{quant}.

%%%%%%%%%%%%%%%%%%%%%%%%%%%%%%%%%%%%%%%%%%%%%%%%%%%%%%%%%%%%%%%%%%%%%%%%%%%%
%%%%%%%%%%%%%%%%%%%%%%%%%%%%%%%%%%%%%%%%%%%%%%%%%%%%%%%%%%%%%%%%%%%%%%%%%%%%%

\section{A technical lemma}
\label{warm-up}

The key result for the proof of Proposition~\ref{prop: ap nice} and its related consequences is the following Lemma.

\begin{lemma}\label{lem}
  Let $S$ be a set of positive integers having positive Banach
  density.  Let $N = [1/\delta(S)] + 1$, where $[ x ]$ as usual
  denotes the greatest integer less than or equal to $x$.  Then there
  is a positive integer $k$ and a subset $Q \subseteq S$ such that
\begin{enumerate}
\item $k \leq N - 1$;
\item $\delta(Q) \ge \frac{ N\delta(S)-1}{2N^2(N-1)}>0$; and
\item for all $a \in Q$, we have $a + k \in S$. 
\end{enumerate}
\end{lemma}
\begin{proof}
   By assumption, $\frac{1}{N}<\delta(S)$. So there exist intervals $I_n$
  with $|I_n|\to\infty$ such that $$\frac{|S\cap
    I_n|}{|I_n|}>\frac{\delta(S)+\frac{1}{N}}{2}.$$ Let $P=\left\{i\colon
    \left|\{iN+1,\ldots ,(i+1)N\}\cap S\right|\ge 2\right\}$.  We
  claim that $P$ has positive Banach density.  To see this, let
  $J_n=\{i\colon \{iN+1,\ldots ,(i+1)N\}\subseteq I_n\}$.  Then
  $|J_n|\le \frac{|I_n|}{N}$ and $|J_n| \to \infty$ as $n\to \infty$.
  For $i\in J_n\setminus P$ we have $S\cap \{iN+1,\ldots ,(i+1)N\}$
  has size at most $1$ and for $i\in P\cap J_n$ we have $S\cap
  \{iN+1,\ldots ,(i+1)N\}$ has size at most $N$.  Since there are at
  most $2N$ elements of $I_n$ that are not accounted for by taking the
  union of the $\{iN+1,\ldots ,(i+1)N\}$ with $i\in J_n$, we see that
$$\frac{\left(\delta(S)+\frac{1}{N}\right)\cdot |I_n|}{2} \le |I_n\cap S| \le |J_n\setminus P| + N|P\cap J_n| + 2N.$$
Using the fact that $|J_n|\le \frac{|I_n|}{N}$, we see
$$\frac{\left(\delta(S)+\frac{1}{N}\right)\cdot N|J_n|}{2} \le |J_n\setminus P| + N|P\cap J_n| + 2N.$$  Dividing by $N|J_n|$ now gives
$$\frac{\delta(S)+\frac{1}{N}}{2}\le \frac{|J_n\setminus P|}{N|J_n|} + \frac{|P\cap J_n|}{|J_n|} + \frac{2}{|J_n|}.$$  Since $|J_n\setminus P|\le |J_n|$, we get
$$\frac{\delta(S)+\frac{1}{N}}{2}\le \frac{1}{N} + \frac{|P\cap J_n|}{|J_n|} + \frac{2}{|J_n|},$$ which gives
$$\frac{\delta(S)-\frac{1}{N}}{2}\le \frac{|P\cap J_n|}{|J_n|} +
\frac{2}{|J_n|}.$$ Since $|J_n|\to \infty$, we see that $\delta(P) \geq
\frac{\delta(S)-\frac{1}{N}}{2}$.  
 
For each $i\in P$, we pick $a_i,b_i\in \{iN+1,\ldots ,(i+1)N\}\cap S$ with $0<b_i-a_i<N$.
For $j\in \{1,\ldots ,N-1\}$, we let $P_j:=\{i\in P\colon b_i-a_i=j\}$.  Then $P=\cup_{j=1}^{N-1} P_j$ and since Banach density is subadditive, we have
$$\delta(P)\le \sum_{j=1}^{N-1} \delta(P_j).$$  Thus there is some
$k\in \{1,\ldots ,N-1\}$ such that $\delta(P_k)\geq \frac{\delta(P)}{N-1}$.  
Let $Q:=\{a_i\colon i\in P_k\}\subseteq S$.  Then $a + k \in S$ for all $a \in Q$
and a simple computation yields $$\delta(Q) \geq \frac{\delta(P_k)}{N} \ge \frac{N\delta(S)-1}{2N^2(N-1)}>0,$$ as desired.  
\end{proof}

We find useful (see Remark~\ref{remark reccursive formula}) stating the following Corollary of Lemma~\ref{lem}.
\begin{cor}
\label{corollary for N}
Let $S$ be a set of positive integers having positive Banach
  density.   Then there
  is a positive integer $k<\frac{2}{\delta(S)}$ and a subset $Q \subseteq S$ such that
\begin{enumerate}
\item[(a)] $\delta(Q) \ge \frac{ \delta(S)^3}{24}$; and
\item[(b)] for all $a \in Q$, we have $a + k \in S$. 
\end{enumerate}
\end{cor}

\begin{proof}
We let $\delta:=\delta(S)$, and we apply Lemma~\ref{lem} with $N=\left[\frac{2}{\delta}\right]$ (which is at least equal to $\left[\frac{1}{\delta}\right]+1$ since $\delta\le 1$). This shows the existence of a set $Q\subseteq S$ satisfying property (b) above; in addition $\delta(Q)\ge \frac{N\delta -1}{2N^2(N-1)}$. 
So, in order to show that (a) holds, it suffices to prove that $ \frac{N\delta -1}{2N^2(N-1)}\ge \frac{\delta^3}{24}$, which is equivalent with proving that
$$\frac{N\delta-1}{N-1}\ge \frac{N^2\delta^3}{12}=\frac{2}{3N}\cdot \left(\frac{N\delta}{2}\right)^3.$$
Since $N=\left[\frac{2}{\delta}\right]\le \frac{2}{\delta}$, then it suffices to show that $\frac{N\delta-1}{N-1}\ge\frac{2}{3N}$, which is equivalent with showing that 
\begin{equation}
\label{necessary inequality}
\delta \ge \frac{5}{3N}-\frac{2}{3N^2}. 
\end{equation}
Because $N=\left[\frac{2}{\delta}\right]$, then $\frac{2}{\delta}<N+1$ and so, $\delta > \frac{2}{N+1}$. Then inequality \eqref{necessary inequality} follows since
\begin{equation}
\label{necessary inequality 2}
\frac{2}{N+1}-\left(\frac{5}{3N}-\frac{2}{3N^2}\right)=\frac{N-5}{3N(N+1)}+\frac{2}{3N^2} \ge 0.
\end{equation}
Inequality~\ref{necessary inequality 2} is obvious for all $N\ge 5$, while for $N\in\{2,3,4\}$ the inequality can be checked directly (note that $N=\left[\frac{2}{\delta}\right]\ge 2$ because $\delta\le 1$). 
\end{proof}

%#############################################################################
%##############################################################################

\section{Proof of our main results}
\label{proof of our main result}

Theorem~\ref{main result} will follow as a consequence of the
following Proposition which is a generalization of
Proposition~\ref{prop: ap nice}.  

\begin{prop}
\label{prop: ap}
Let $X$ be a Noetherian topological space, let $\Phi:X\lra X$ be a continuous map, let $x\in X$, let $Y$ be a closed subset of $X$, and let $S:=\{n\colon \Phi^n(x)\in Y\}$. If $S$ has positive Banach density, then it contains an infinite arithmetic progression.
\end{prop}

\begin{proof} 
Consider the set $\mathcal{V}$ of all closed subsets $V$ of $X$ with
the property that $T_V:=\{n\colon \Phi^n(x)\in V\}$ has positive
Banach density but does not contain an infinite arithmetic
progression.  If $\mathcal{V}$ is empty, then there is nothing to
prove.  Thus we may assume, towards a contradiction, that
$\mathcal{V}$ is non-empty.  We let $W$ be a minimal element of
$\mathcal{V}$ with respect to the inclusion of sets (note that such an
element exists since $X$ is Noetherian).  By Lemma \ref{lem}, we have
a positive integer $k$ and a  
subset $Q \subseteq T_W$ with $\delta(Q) > 0$ such that $a + k \in
T_W$ for all $a \in Q$.  

If $n\in Q$, then $\Phi^n(x)\in W$ and $\Phi^{n+k}(x)\in W$.  Thus $\Phi^n(x)\in W\cap \Phi^{-k}(W)$ whenever $n\in Q$.  If $\Phi^{-k}(W)\supseteq W$ then $T_W$ has the property that $n+k\in T_W$ whenever $n\in T_W$ and since $T_W$ is non-empty, it contains an infinite arithmetic progression, which contradicts the fact that $W\in \mathcal{V}$.   Thus $Z:=W\cap \Phi^{-k}(W)$ is a proper closed subset of $W$ (since $\Phi$ is continuous and $W$ is closed) and so we have $\Phi^n(x)\in Z$ for all $n\in Q$.  Since $Q$ has positive Banach density, we obtain that $T_Z\supseteq Q$ also has positive Banach density and therefore $T_Z$ contains an infinite arithmetic progression.  Since $T_Z\subseteq T_W$, we see that $T_W$ contains an infinite arithmetic progression, a contradiction. This concludes our proof.
\end{proof}

Theorem~\ref{main result} follows easily now.

\begin{proof}[Proof of Theorem \ref{main result}.] Suppose not.  Let $\mathcal{V}$ be the collection of all closed subsets $V$ of $X$ that have the property that there is some continuous map $g:V\to V$, and some closed subset $W$ of $V$, and a point $y\in V$ such that  $\{n\colon g^{n}(y)\in W\}$ cannot be expressed as a finite union of arithmetic progressions along with a set of Banach density zero.

By assumption, $X\in \mathcal{V}$ and so we may choose a minimal element $V\in \mathcal{V}$.  Then there is some continuous map $g:V\lra V$, some closed subset $W$ of $V$, and some point $y\in V$ such that  $T:=\{n\colon g^{n}(y)\in W\}$ cannot be expressed as a finite union of arithmetic progressions along with a set of Banach density zero.  We necessarily have that $W_i:=g^{-i}(W)$ is a proper closed subset of $V$ (note that $g$ is continuous), since otherwise $T$ would contain every integer greater than or equal to $i$ (and thus it would be the union of an arithmetic progression with a finite set).  Moreover, by our choice of $V$, $W$ and $y$, it follows that $\delta(T)>0$ and thus by
Proposition \ref{prop: ap}, there exist $a,b\in \mathbb{N}$ such that $T\supseteq \{an+b\colon n\ge 0\}$.  Let $C_i$ denote the closure of $S_i:=\{g^{(an+b)}(y)\colon n\ge i\}$.  Then $$C_0\supseteq C_1\supseteq \cdots $$ is a descending chain of closed sets and hence there is some $m$ such that $C_m=C_{m+1}=\cdots$.   We take $V_0=C_m$.  Then $g^{-a}(V_0)\supseteq g^{-a}(S_{m+1})\supseteq S_m$ and since $g^{-a}(V_0)$ is closed we thus see it contains the closure of $S_m$, which is $V_0$.

Then $V_0\subseteq W$ is closed and we have $g^{-a}(V_0)\supseteq V_0$.  
We let $V_j$ denote the closed set $g^{-j}(V_0)$ for $j\in \{1,\ldots ,a-1\}$.  Since $V_j\subseteq W_{a+j}\subsetneq V$, we see that each $V_j$ is a proper  subset for $0\le j\le a-1$.  Then $g^{-a}(V_j)=g^{-a}(g^{-j}(V_0))=g^{-j}(g^{-a}(V_0)) \supseteq g^{-j}(V_0)=V_j$, and so for $j\in \{0,\ldots ,a-1\}$, we have $g^{-j+na+b}(y)\in V_{j}$ for every $n> m$.  Moreover, since $g^{-a}(V_j)\supseteq V_j$, we have that $h:=g^a$ restricts to continuous maps $h:V_j\lra V_j$ for each $0\le j\le a-1$.  We let $y_j:=g^{-j+a+b}(y)$. It follows from the minimality of $V$ that 
$$T_j:=\{n\ge m\colon h^{n}(y_j)\in W\cap V_j\}$$ is a finite union of arithmetic progressions along with a set of Banach density zero. On the other hand, 
$$T_j=\{n\ge m\colon g^{-j+a(n+1)+b}(y)\in W\},$$
for each $j=0,\cdots,a-1$. 
Then, up to a finite set, we have $$T= \bigcup_{j=0}^{a-1} (aT_j+b+a-j),$$ where for any set $U\subseteq \N$ and any $c\in\N$, we let $c\cdot U$ be the set $\{cj\colon j\in U\}$, and we let $c+U:=U+c$ be the set $\{c+j\colon j\in U\}$.  Hence $T$ is a finite union of arithmetic progressions along with a set of Banach density zero.
\end{proof}

As a consequence to our main result, we can prove the following (seemingly) stronger statement.
\begin{theorem}
\label{most general result}
Let $X$ be a Noetherian space, let $U\subseteq X$ be an open subset, and let $\Phi:U\lra X$ be a continuous map. Let $x\in X$ such that $\Phi^n(x)\in U$ for each nonnegative integer $n$. Then for each closed set $Y\subseteq X$, the set $S:=\{n\in\N\colon \Phi^n(x)\in Y\}$ is a union of at most finitely many arithmetic progressions along with a set of Banach density zero
\end{theorem}

\begin{proof}
Let $Z:=\cap_{n\ge 0}\Phi^{-n}(U)$; we know that $Z$ is nonempty since $x$ (and therefore $\cO_\Phi(x)$) is contained in $Z$. We endow $Z$ with the inherited topology from $X$; then $Z$ is also a Noetherian space. Furthermore, by its definition, we obtain that $\Phi$ restricts to a self-map  $\Phi_Z:Z\lra Z$. Next we show that $\Phi_Z$ is continuous. 

Indeed, let $V\subseteq X$ be an open set and we need to show that $\Phi_Z^{-1}(V\cap Z)$ is open in $Z$. This follows immediately once we prove that $\Phi_Z^{-1}(V\cap Z)=\Phi^{-1}(V)\cap Z$ because $\Phi:U\lra X$ is continuous and so $\Phi^{-1}(V)$ is open in $U$ and (because $U$ is an open subset of $X$) it is also open in $X$ which yields that $\Phi^{-1}(V)\cap Z$ is open in $Z$. To see that $\Phi_Z^{-1}(V\cap Z)=\Phi^{-1}(V)\cap Z$ we note that for each $y\in \Phi_Z^{-1}(V\cap Z)\subseteq Z$ we have $\Phi_Z(y)\in V$. So, $\Phi(y)\in V$ and thus $y\in \Phi^{-1}(V)\cap Z$. Conversely, if $y\in \Phi^{-1}(V)\cap Z$, then $y\in Z$ and so $\Phi_Z(y)\in V\cap Z$ as claimed. 

Therefore $\Phi_Z:Z\lra Z$ is a continuous map on a Noetherian space. Hence by Theorem~\ref{main result}, the set of all $n\in \N$ such that $\Phi_Z^n(x)\in Y\cap Z$ is a union of at most finitely many arithmetic progressions along with a set of Banach density zero. Because $\Phi=\Phi_Z$ on $Z$ then $\Phi^n(x)\in Y$ if and only if $\Phi_Z^n(x)\in Y\cap Z$, which concludes our proof.
\end{proof}

Corollary~\ref{rational maps corollary} is an immediate consequence of Theorem~\ref{most general result}. The proof for Theorem~\ref{main result DMM} is similar to the proof of Theorem~\ref{main result} and it relies on the following result.
\begin{prop}
\label{prop: ap DMM}
Let $X$ be a Noetherian space, let $f:X\lra X$ be a continuous function, let $\{x_{-n}\}_{n\ge 0}$ be a coherent backward orbit of a point  $x\in X$, and let $Y\subseteq X$ be a closed set. If the set $S:=\{n\in\N\colon x_{-n}\in Y\}$ has positive Banach density then it contains an infinite arithmetic progression.
\end{prop}

\begin{proof}
The proof is similar to the proof of Proposition~\ref{prop: ap
  DMM}. Consider the set $\mathcal{V}$ of all closed subsets $V$ of
$X$ with the property that $T_V:=\{n\colon x_{-n}\in V\}$ has positive
Banach density but does not contain an infinite arithmetic
progression.  If $\mathcal{V}$ is empty, then there is nothing to
prove.  Thus we may assume, towards a contradiction, that
$\mathcal{V}$ is non-empty.  We let $W$ be a minimal element of
$\mathcal{V}$ with respect to the inclusion of sets (note that such an
element exists since $X$ is Noetherian). By Lemma~\ref{lem},  we see
that there exists a positive integer $k$ and a set $Q\subseteq T_W$ of positive Banach density such that if $n\in Q$ then $n+k\in T_W$. Thus $x_{-n-k}\in W\cap f^{-k}(W)$ whenever $n\in Q$. There are two cases: either $f^{-k}(W)\supseteq W$ or not.

Assume now that $f^{-k}(W)\supseteq W$; so if $y\in W$, then also $f^k(y)\in W$. Then $T_W$ has the property that $n-k\in T_W$ whenever $n\in T_W$ and $n\ge k$. Because $Q$ has positive Banach density, then it is infinite, and therefore there exists $j\in \{0,\dots, k-1\}$ such that there exist infinitely many $n\in Q$ satisfying $n\equiv j\pmod{ k}$. So there exists a sequence of integers $n_i\to \infty$ contained in $T_W$ such that $n_i\equiv j\pmod{k}$ for each $i$ and moreover, $n_i-\ell k\in T_W$ for all $\ell\ge 0$ such that $n_i-\ell k\ge 0$. Thus $T_W$ contains the infinite arithmetic progression $\{j+\ell k\}_{\ell\ge 0}$, as desired.

Assume now that $W\not\subseteq f^{-k}(W)$.   Then $Z:=W\cap f^{-k}(W)$ is a proper closed subset of $W$ and so we have $x_{-n-k}\in Z$ for all $n\in Q$.  Since $Q$ has positive Banach density, we obtain that also $T_Z$ has positive Banach density (note that $\delta(Q+k)=\delta(Q)>0$ and $(Q+k)\subseteq T_Z$). Hence $T_Z$ contains an infinite arithmetic progression (because $Z$ is a proper subset of $W$).  Since $T_Z\subseteq T_W$, we see that $T_W$ contains an infinite arithmetic progression, a contradiction. This concludes our proof.
\end{proof}

Theorem~\ref{main result DMM} follows from Proposition~\ref{prop: ap DMM} similar to the proof of Theorem~\ref{main result}.

\begin{proof}[Proof of Theorem~\ref{main result DMM}.]
Let $\mathcal{V}$ be the collection of all closed subsets $V$ of $X$ that have the property that there is some continuous map $g:V\to V$, and some closed subset $W$ of $V$, and a coherent backward orbit $\{y_{-n}\}$ of a point $y\in V$ such that  $\{n\colon y_{-n}\in W\}$ cannot be expressed as a finite union of arithmetic progressions along with a set of Banach density zero.

Assume, for contradiction, that $X\in \mathcal{V}$ and so we may choose a minimal element $V\in \mathcal{V}$.  Then there is some continuous map $g:V\lra V$, some closed subset $W$ of $V$, and a coherent backward orbit $\{y_{-n}\}_{n\ge 0}$ of a point $y\in V$ such that  $T:=\{n\colon y_{-n}\in W\}$ cannot be expressed as a finite union of arithmetic progressions along with a set of Banach density zero.  

For each $i\in\N$ we let $W_i:=g^{-i}(W)$. We claim that if $W_i=V$ for some $i\in\N$, then $T$ is a union of at most finitely many arithmetic progressions along with a finite set (which obviously has Banach density zero). Indeed, for each $j\in\{0,\dots, i-1\}$ we let $T_{i,j}:=\{n\in T\colon n\equiv j\pmod{i}\}$. Then $T=\cup_{j=0}^{i-1}T_{i,j}$. Assume $T_{i,j}$ is infinite (for some $0\le j\le i-1$). Since $W_i=V$ then for each $n\in T$ also $n-i\in T$ and therefore $n-i\ell\in T$ for all $\ell\ge 0$ such that $n-i\ell\ge 0$. Because we assumed that $T_{i,j}$ is infinite, then $j+i\ell\in T$ for all $\ell\ge 0$. In conclusion, $T$ is indeed a union of at most finitely many arithmetic progressions (of ratio $i$) along with a finite set.

So from now on assume that $W_i:=g^{-i}(W)$ is a proper closed subset of $V$ (note that $g$ is continuous).  Moreover, by our choice of $V$, $W$ and $y$, it follows that $\delta(T)>0$ and thus by
Proposition \ref{prop: ap DMM}, there exist $a,b\in \mathbb{N}$ such that $T\supseteq \{an+b\colon n\ge 0\}$.  Let $C_i$ denote the closure of $S_i:=\{y_{-(an+b)}\colon n\ge i\}$.  Then $$C_0\supseteq C_1\supseteq \cdots $$ is a descending chain of closed sets and hence there is some $m$ such that $C_m=C_{m+1}=\cdots$.   We take $V_0=C_m$.  Then $g^{-a}(V_0)\supseteq g^{-a}(S_{m})\supseteq S_{m+1}$ and since $g^{-a}(V_0)$ is closed we thus see it contains the closure of $S_{m+1}$, which is $V_0$.

Then $V_0\subseteq W$ is closed and we have $g^{-a}(V_0)\supseteq V_0$.  
We let $V_j$ denote the closed set $g^{-j}(V_0)$ for $j\in \{1,\ldots ,a-1\}$.  Since $V_j\subseteq W_{a+j}\subsetneq V$, we see that each $V_j$ is a proper  subset for $0\le j\le a-1$.  Then $g^{-a}(V_j)=g^{-a}(g^{-j}(V_0))=g^{-j}(g^{-a}(V_0)) \supseteq g^{-j}(V_0)=V_j$, and so for $j\in \{0,\ldots ,a-1\}$, we have $y_{-(na+b+j)}\in V_{j}$ for every $n\ge m$.  Moreover, since $g^{-a}(V_j)\supseteq V_j$, we have that $h:=g^a$ restricts to continuous maps $h:V_j\lra V_j$ for each $0\le j\le a-1$.  It follows from the minimality of $V$ that 
$$T_j:=\{n\ge m\colon (y_{-(na+b+j)})\in W\cap V_j\}$$ is a finite union of arithmetic progressions along with a set of Banach density zero (since $\{y_{-j-na-b}\}_{n\ge 0}$ is a coherent backward orbit of $y_{-j-b}$ under the action of $h$). 
Then, up to a finite set, we have $$T= \bigcup_{j=0}^{a-1} (aT_j+b+j)$$ and hence $T$ is a finite union of arithmetic progressions along with a set of Banach density zero.
\end{proof}

\section{Some quantitative results}\label{quant}

The following result is an easy application of Lemma~\ref{lem}.

\begin{theorem}
\label{the case of curves}
Let $X$ be a quasi-projective variety defined over a field $K$, let $\Phi:X\lra X$ be an endomorphism defined over $K$,  let $C\subseteq X$ be an irreducible curve, and let $\alpha \in X(K)$ be a point that is not
preperiodic under $\Phi$.    
If the set $S:=\{n\in\N\colon \Phi^n(\alpha) \in C(K)\}$ has Banach
density $\delta>0$, then $S$ contains an infinite arithmetic
progression of ratio $k = \frac{1}{\delta}$, and $\Phi^k(C)\subseteq C$.
\end{theorem}
\begin{proof}
  It follows from Lemma~\ref{lem} applied with
  $N=\left[\frac{1}{\delta}\right]+1$ that there exists a positive
  integer $k\le \left[\frac{1}{\delta}\right]$ and a subset $Q\subset
  S$ of positive density such that for each $n\in Q$, also
  $\Phi^{n+k}(\alpha)\in C(K)$. So $\Phi^n(\alpha)\in C\cap
  \Phi^{-k}(C)$.  Hence, $C\cap\Phi^{-k}(C)$ contains an infinite set
  of points. Since $C$ is an irreducible curve, we see then that
  $C\subseteq \Phi^{-k}(C)$; thus $\Phi^k(C)\subseteq C$. This yields the desired infinite
  arithmetic progression of ratio $k\leq \frac{1}{\delta}$.  If $k <1/\delta$, then the existence of this arithmetic progression would imply that
  $\delta \geq 1/k > \delta$, a contradiction.  Thus, $k =
  \frac{1}{\delta}$.
\end{proof}

\begin{remark}
\label{remark reccursive formula}
Applying the technique of the proof of Theorem~\ref{the case of
  curves} recursively in the case of endomorphisms $\Phi$ of $\bP^n$
of degree $m$ one can obtain a similar result for all projective
irreducible subvarieties $V\subseteq \bP^n$ of degree at most $D$ and
dimension at most $e$. More precisely, with the above geometric data $m,D,e$,
for every $\delta>0$ there exists a bound $M:=M(\delta, m, D,e)$ such
that if the set $S:=\{n\in\N\colon \Phi^n(\alpha)\in V(K)\}$ has
Banach density at least equal to $\delta$, then $S$ contains an
infinite arithmetic progression of ratio at most $M$. One uses induction on
the dimension $e$ of $V$, with the base case being Theorem~\ref{the
  case of curves}. Then with the use of B\'ezout's Theorem, one
controls both the degrees and the number of the irreducible components
of $V\cap \Phi^{-k}(V)$. We obtain $M(\delta,m,D,1)=\frac{1}{\delta}$ (by Theorem~\ref{the case of curves}), and then for $e\ge 2$ one  applies this technique coupled with Corollary~\ref{corollary for N} to derive the following recursive formula
$$M(\delta,m,D,e)=M\left(\frac{\delta^3}{24m^{\frac{2}{\delta}}D^2}, m, m^{\frac{2}{\delta}}D^2, e-1\right).$$
\end{remark}

%##############################################################################
%##############################################################################


\begin{thebibliography}{9}
\newcommand{\au}[1]{{#1},}
\newcommand{\ti}[1]{\textit{#1},}
\newcommand{\jo}[1]{{#1}}
\newcommand{\vo}[1]{\textbf{#1}}
\newcommand{\yr}[1]{(#1),}
\newcommand{\page}[1]{#1.}
\newcommand{\ppx}[1]{#1,}
\newcommand{\pps}[1]{#1.}
\newcommand{\bk}[1]{{#1},}
\newcommand{\inbk}[1]{in: {#1}}
\newcommand{\xxx}[1]{{arXiv:}}

\bibitem{JPB}
J.~P.~Bell, \emph{A generalized Skolem-Mahler-Lech theorem for affine varieties}, J. London Math. Soc. (2) \textbf{73} (2006), 367--379.

\bibitem{AJM}
J.~P.~Bell, D.~Ghioca, and T.~J.~Tucker, \emph{The dynamical Mordell-Lang problem for \'{e}tale maps}, Amer. J. Math. \textbf{132} (2010), 1655--1675.

\bibitem{Bell-Lagarias}
J.~P.~Bell and J.~Lagarias, \emph{A Skolem-Mahler-Lech theorem for iterated automorphisms of $K$-algebras}, to appear in the Canad. J. Math., 2014.

\bibitem{gap-Compo}
\au{R.~L.~Benedetto, D.~Ghioca, P.~Kurlberg, and T.~J.~Tucker}
\ti{A gap principle for dynamics}
\jo{Compositio Math.}
\yr{2010}
\vo{146}
\pps{1056--1072}

\bibitem{BGKTZ}
R.~L.~Benedetto, D.~Ghioca, P.~Kurlberg, and T.~J.~Tucker, \emph{A case of the dynamical Mordell-Lang conjecture} (with an Appendix by U.~Zannier), Math. Ann. \textbf{352} (2012), 1--26.

\bibitem{BGHKST}
\au{R.~L.~Benedetto, D.~Ghioca, B.~A.~Hutz, P.~Kurlberg, T.~Scanlon, and T.~J.~Tucker}
\ti{Periods of rational maps modulo primes}
\jo{Math. Ann.}
\vo{355}
\yr{2013}
\pps{637--660}

\bibitem{Denis}
\au{L.~Denis}
\ti{G\'eom\'etrie et suites r\'ecurrentes}
\jo{Bull. Soc. Math. France}
\vo{122}
\yr{1994}
\pps{13--27}

\bibitem{Fakhruddin}
N.~Fakhruddin, \emph{The algebraic dynamics of generic endomorphisms of $\bP^n$}, to appear in Algebra \& Number Theory, 2013.

\bibitem{Faltings}
\au{G.~Faltings}
\ti{The general case of S. Lang's conjecture}
Barsotti Symposium in Algebraic Geometry (Abano Terme, 1991), 175--182, Perspect. Math. \textbf{15}, Academic Press, 1994.

%\bibitem{Favre}
%C.~Favre, \emph{Dynamique des applications rationnelles}, PhD thesis, 2002.

\bibitem{GT-Compositio}
\au{D.~Ghioca and T.~J.~Tucker}
\ti{A dynamical version of the Mordell-Lang conjecture for the additive group}
\jo{Compositio Math.}
\vo{164}
\yr{2008}
\pps{304--316}

\bibitem{GT-JNT}
\au{D.~Ghioca and T.~J.~Tucker}
\ti{Periodic points, linearizing maps, and the dynamical Mordell-Lang problem}
\jo{J. Number Theory}
\vo{129}
\yr{2009}
\pps{1392--1403}

\bibitem{GT-IMRN}
\au{D.~Ghioca, T.~J.~Tucker and S.~Zhang}
\ti{Towards a dynamical Manin-Mumford Conjecture}
\jo{IMRN}
\vo{22}
\yr{2011}
\pps{5109--5122} 


\bibitem{GTZ-1}
\au{D.~Ghioca, T.~J.~Tucker, and M.~E.~Zieve}
\ti{Intersections of polynomial orbits, and a dynamical Mordell-Lang conjecture}
\jo{Invent. Math.}
\vo{171}
\yr{2008}
\pps{463--483}

\bibitem{JNT-Bordeaux}
\au{D.~Ghioca, T.~J.~Tucker, and M.~E.~Zieve}
\ti{The Mordell-Lang question for endomorphisms of semiabelian varieties}
\jo{J. Theor. Nombres Bordeaux}
\vo{23}
\yr{2011}
\pps{645--666}



\bibitem{GTZ-2}
\au{D.~Ghioca, T.~J.~Tucker, and M.~E.~Zieve}
\ti{Linear relations between polynomial orbits}
\jo{Duke Math. J.}
\yr{2012}
\vo{161}
\pps{1379--1410}

\bibitem{Gignac}
W.~Gignac, \emph{Measures and dynamics on Noetherian spaces}, to appear in J. Geom. Anal.

\bibitem{Fabien}
\au{F.~Pazuki}
\ti{Zhang's conjecture and squares of abelian surfaces}
\jo{C. R. Math. Acad. Sci. Paris}
\yr{2010}
\vo{348}
\pps{483--486}

\bibitem{Clay}
C.~Petsche, \emph{On the distribution of orbits in affine varieties}, preprint, {\xxx 11401.3425.}


\bibitem{Vojta}
\au{P.~Vojta}
\ti{Integral points on subvarieties of semiabelian varieties. I}
\jo{Invent. Math.}
\vo{126}
\yr{1996}
\pps{133--181}

\bibitem{Wang}
M.~Wang, \emph{Rational points and transcendental points}, PhD thesis, ETH Zurich, 2011. 

\bibitem{Xie}
J.~Xie, \emph{Dynamical Mordell-Lang Conjecture for birational polynomial morphisms on $\bA^2$},  preprint, {\xxx 11303.3631.}

\bibitem{ZhangLec}
S.~Zhang, \emph{Distributions in algebraic dynamics}, Survey in
  Differential Geometry, vol.~10, International Press, 2006, pp.~381--430.

\end{thebibliography}
\end{document}